\documentclass[12pt]{amsart}
\usepackage{amsmath,latexsym,amsfonts,amssymb,amsthm}
\usepackage{geometry}
\usepackage{hyperref}
\usepackage{mathrsfs}
\usepackage{graphicx,color}
\usepackage{tikz-cd}
\usepackage{tikz}
\usetikzlibrary{positioning}
\usepackage{mathtools}
\usepackage[all,cmtip]{xy}
\usepackage{enumitem}

\newtheorem{prop}{Proposition}[section]
\newtheorem{thm}[prop]{Theorem}
\newtheorem{cor}[prop]{Corollary}
\newtheorem{conj}[prop]{Conjecture}
\newtheorem{lem}[prop]{Lemma}

\theoremstyle{definition}

\newtheorem{defn}[prop]{Definition}

\newtheorem{rem}[prop]{\it Remark}

\newtheorem*{claim*}{Claim}
\newtheorem*{psr}{Postscript Remark}

\newcommand{\bP}{\mathbb{P}}
\newcommand{\bC}{\mathbb{C}}
\newcommand{\bR}{\mathbb{R}}
\newcommand{\bA}{\mathbb{A}}
\newcommand{\bQ}{\mathbb{Q}}
\newcommand{\bZ}{\mathbb{Z}}

\newcommand{\bG}{\mathbb{G}}

\newcommand{\cX}{\mathcal{X}}

\newcommand{\cO}{\mathcal{O}}

\newcommand{\cM}{\mathcal{M}}

\newcommand{\cN}{\mathcal{N}}
\newcommand{\cR}{\mathcal{R}}
\newcommand{\cD}{\mathcal{D}}
\newcommand{\cH}{\mathcal{H}}

\newcommand{\fa}{\mathfrak{a}}
\newcommand{\fb}{\mathfrak{b}}

\newcommand{\fm}{\mathfrak{m}}
\newcommand{\fB}{\mathfrak{B}}
\newcommand{\fp}{\mathfrak{p}}
\newcommand{\fq}{\mathfrak{q}}

\newcommand{\Spec}{\mathrm{Spec}}
\newcommand{\Supp}{\mathrm{Supp}~}

\newcommand{\mult}{\mathrm{mult}}
\newcommand{\lct}{\mathrm{lct}}

\newcommand{\vol}{\mathrm{vol}}
\newcommand{\ord}{\mathrm{ord}}

\newcommand{\hvol}{\widehat{\vol}}
\newcommand{\hell}{\widehat{\ell}}
\newcommand{\rme}{\mathrm{e}}
\newcommand{\hcol}{\widehat{\ell}}
\newcommand{\e}{\mathrm{e}}
\newcommand{\bin}{\mathbf{in}}
\newcommand{\Val}{\mathrm{Val}}
\newcommand{\lnlc}{\ell_{\mathrm{nlc}}}
\newcommand{\lnklt}{\ell_{\mathrm{nklt}}}

\begin{document}

\title[Birational superrigidity and K-stability]{Birational superrigidity and K-stability of singular Fano complete intersections}

\author{Yuchen Liu}
\address{Department of Mathematics, Yale University, New Haven, CT 06511, USA.}
\email{yuchen.liu@yale.edu}

\author{Ziquan Zhuang}
\address{Department of Mathematics, Princeton University, Princeton, NJ 08544, USA.}
\email{zzhuang@math.princeton.edu}
\date{}

\subjclass[2010]{14E08, 32Q20 (primary), 13A18, 14M10 (secondary).}

\maketitle

\begin{abstract}
    We introduce an inductive argument for proving birational superrigidity and K-stability of singular Fano complete intersections of index one, using the same types of information from lower dimensions. In particular, we prove that a hypersurface in $\mathbb{P}^{n+1}$ of degree $n+1$ with only ordinary singularities of multiplicity at most $n-5$ is birationally superrigid and K-stable if $n\gg0$. As part of the argument, we also establish an adjunction type result for local volumes of singularities.
\end{abstract}

\section{Introduction}

Due to their close relations to the non-rationality of Fano varieties and the existence of K\"ahler-Einstein metric, the notions of birational superrigidity and K-(semi)stability have generated a lot of interest. However, it can be challenging to determine whether a given Fano variety is birationally superrigid and K-(semi)stable or not. The case of smooth complete intersections of index one is better understood, thanks to the work of \cite{IM-quartic,Puk-general-hypersurface,Che-quintic,dFEM-bounds-on-lct,dF-hypersurface,Suzuki-cpi,Fujita-alpha,Z-cpi}, culminating into the following theorem.

\begin{thm}[\emph{loc. cit.}] \leavevmode
    \begin{enumerate}
        \item Every smooth hypersurface $X\subseteq \bP^{n+1}$ of degree $n+1$ is K-stable and is birationally superrigid when $n\ge 3$;
        \item Every smooth Fano complete intersection $X\subseteq\bP^{n+r}$ of index $1$, codimension $r$ and dimension $n\ge 10r$ is birationally superrigid and K-stable.
    \end{enumerate}
\end{thm}

When the Fano varieties are singular, the situation is less clear, even for hypersurfaces $X\subseteq\bP^{n+1}$ of index one with only ordinary singularities (i.e. isolated singularities whose projective tangent cone is smooth). In general, they fail to be birationally superrigid or K-stable when they are too singular. For example, $X$ is rational (resp. admits a birational involution) if it has a point of multiplicity $n$ (resp. $n-1$), and therefore is not birationally superrigid. By the work of the first author \cite{Liu18}, we also know that a Fano variety is not K-(semi)stable if the local volume at a singular point is too small. On the other hand, it is expected that the index one hypersurfaces $X$ are birationally superrigid and K-stable when their singularities are mild. Indeed, Pukhlikov \cite{Puk-singular} shows that a \emph{general} hypersurface $X\subseteq\bP^{n+1}$ of degree $n+1$ having a singular point of multiplicity $\mu\le n-2$ is birationally superrigid (see also \cite{EP-sing-cpi} which treats general complete intersections whose projective tangent cones are intersections of hyperquadrics), de Fernex \cite{dF-singular} establishes their birational superrigidity under certain numerical assumption (involving the dimension of singular locus and the Jacobian ideals of linear space sections of $X$) which in particular applies to index one hypersurfaces $X\subseteq\bP^{n+1}$ with ordinary singularities of multiplicity at most $2\sqrt{n+1}-7$ and Suzuki \cite{Suzuki-cpi} generalizes his method to certain singular complete intersections. As for K-stability, very little is known except in small dimensions or small degrees, e.g. log del Pezzo surfaces \cite{OSS-del-pezzo}, quasi-smooth 3-folds \cite{JK-weighted-3fold}, complete intersection of two quadric hypersurfaces in all dimensions \cite{Spotti-Sun}, and cubic threefolds with isolated $A_k\, (k\leq 4)$ singularities \cite{LX-cubic-3fold}.

The main purpose of this paper is to introduce an argument that can be used to prove both birational superrigidity and K-stability for a large class of singular Fano complete intersections of Fano index one (i.e. $-K_X$ is linearly equivalent to the hyperplane class). In particular, we have the following results.

\begin{thm} \label{thm:hypersurface-ordinary}
There exists an absolute constant $N\le 250$ such that every degree $n+1$ hypersurface $X\subseteq\bP^{n+1}$ with only ordinary singularities of multiplicity at most $n-5$ is birationally superrigid and K-stable when $n\ge N$.
\end{thm}

\begin{thm} \label{thm:cpi-ordinary}
Let $\delta\ge -1$ and $r\ge 1$ be integers, then there exists a constant $N$ depending only on $\delta$ and $r$ such that if $X\subseteq\bP^{n+r}$ is a Fano complete intersection of Fano index $1$, codimension $r$ and dimension $n\ge N$ such that 
    \begin{enumerate}
        \item The singular locus of $X$ has dimension at most $\delta$, 
        \item Every $($projective$)$ tangent cone of $X$ is a Fano complete intersection of index at least $4r+2\delta+2$ and is smooth in dimension $r+\delta$,
    \end{enumerate}
then $X$ is birationally superrigid and K-stable.
\end{thm}

\begin{thm} \label{thm:hypersurface-general}
Let $\delta\ge -1$ and $m\ge 1$ be integers. Let $X\subseteq\bP^{n+1}$ be a hypersurface of degree $n+1$. Assume that 
    \begin{enumerate}
        \item The singular locus of $X$ has dimension at most $\delta$ and $n\ge \frac{3}{2}\delta+1$;
        \item $X$ has multiplicity $\e(x,X)\le m$ at every $x\in X$ and the corresponding tangent cone is smooth in codimension $\rme(x,X)-1$;
        \item $\frac{1}{m^m}\cdot \frac{(n-2-\delta)^{n-2-\delta}}{(n-2-\delta)!}\ge \binom{n+2}{\delta+3}$.
    \end{enumerate}
Then $X$ is birationally superrigid and K-stable.
\end{thm}

It may be interesting to look at the asymptotics of Theorem \ref{thm:hypersurface-general} as the dimension gets large. By Stirling's formula, the condition (3) can be replaced by following weaker inequality (at least when $n-\delta \gg 0$):
\[e^{n-2-\delta}\ge m^m (n+2)^{\delta+3}.
\]
After taking the logarithm, it is not hard to see that this inequality is satisfied if $\delta,m\le Cn^{1-\epsilon}$ for some fixed constants $C,\epsilon>0$ and $n\gg 0$. Hence our conditions are asymptotically better than those of \cite{dF-singular}.

Our proof of these three results has close ties to some local invariants of the singularities, i.e. the minimal non-klt (resp. minimal non-lc) colengths defined as follows.

\begin{defn}[\cite{Z-cpi}]
 Let $x\in (X,D)$ be a klt singularity. The minimal non-klt (resp. non-lc) colength of $x\in (X,D)$ is defined as
 \begin{eqnarray*}
    \lnklt(x,X,D) & := & \min\{\ell(\cO_X/\fa)\;|\;\Supp(\cO_X/\fa)=\{x\}\;\mathrm{and}\;(X,D;\fa)\;\mathrm{is}\;\mathrm{not}\;\mathrm{klt}\} \\
    (\mathrm{resp.}\quad \lnlc(x,X,D) & := & \min\{\ell(\cO_X/\fa)\;|\;\Supp(\cO_X/\fa)=\{x\}\;\mathrm{and}\;(X,D;\fa)\;\mathrm{is}\;\mathrm{not}\;\mathrm{lc}\}).
 \end{eqnarray*}
 When $D=0$, we use the abbreviation $\lnklt(x,X)$ (resp. $\lnlc(x,X)$).
\end{defn}

These local invariants govern the birational superrigidity and K-stability of index one complete intersection in a certain sense. More precisely, we prove the following criterion.

\begin{thm} \label{thm:main criterion}
Let $\delta\ge -1$ be an integer and let $X\subseteq\bP^{n+r}$ be a Fano complete intersection of index $1$, codimension $r$ and dimension $n$. Assume that
    \begin{enumerate}
        \item The singular locus of $X$ has dimension at most $\delta$ and $n\ge 2r+\delta+2$;
        \item For every $x\in X$ and every general linear space section $Y\subseteq X$ of codimension $2r+\delta$ containing $x$, we have $\lnlc(x,Y)> \binom{n+r+1}{2r+\delta+1}$.
    \end{enumerate}
Then $X$ is birationally superrigid. If in addition,
    \begin{enumerate}
        \setcounter{enumi}{2}
        \item For every $x\in X$ and every general linear space section $Y\subseteq X$ of codimension $r+\delta$ containing $x$, we have $\lnklt(x,Y)> \binom{n+r}{r+\delta}$,
    \end{enumerate}
then $X$ is also K-stable.
\end{thm}

When the singularities are mild, the corresponding minimal non-klt (resp. non-lc) colengths tend to be large and indeed they are often exponential in $n$ for a given singularity type. In particular, they grow faster than any polynomial in $n$ and the above criterion automatically applies to yield the birational superrigidity and K-stability of many complete intersections in large dimension. As such, the core of our argument consists of finding suitable lower bounds of minimal non-klt (resp. non-lc) colengths for the various singularities we encounter. This can be done in several different ways by relating these local invariants to the conditional birational superrigidity (as we do for Theorem \ref{thm:hypersurface-ordinary} and \ref{thm:cpi-ordinary}) or K-stability (as in the proof of Theorem \ref{thm:hypersurface-general}) of the tangent cone of the singular point, which is a Fano variety of smaller dimension. Hence essentially we obtain an inductive argument for proving birational superrigidity and K-stability using lower dimensional information.

A natural lower bound of the minimal non-klt (non-lc) colength is given by the local volume and normalized colength (see \S \ref{sec:local invariants} for notation and definition) of the singularity, as introduced in \cite{Li-normvol,bl-semicont}. While these numbers are hard to compute in general, we have an explicit formula for the cone over a K-semistable base as, by the work of \cite{LL-vol-min,lx16}, the local volume is computed by the blowup of the vertex in this case. By degeneration and semi-continuity argument, we then get an estimate of local volumes for certain ordinary singularities (indeed, a conjectural lower bound for every ordinary singularity as well). For general singularities, we may try to reduce to this case by taking hyperplane sections. Therefore, a key step in our approach is a comparison between the local volumes of the singularity and its hypersurface section.

\begin{thm}\label{thm:adjunction}
 Let $x\in (X,D)$ be a klt singularity of dimension $n$.
 Let $H$ be a normal reduced Cartier divisor of $X$ containing
 $x$. Assume that $H$ is not an irreducible component
 of $D$. Then we have
 \[
  \frac{\hvol(x,X,D)}{n^n}\geq \frac{\hvol(x,H,D|_H)}{(n-1)^{n-1}},
  \quad \frac{\hcol(x,X,D)}{n^n/n!}\geq \frac{\hcol(x,H,D|_H)}{(n-1)^{n-1}/(n-1)!}.
 \]
\end{thm}

One may interpret this as saying that ``local volume (resp. normalized colength) density'' does not increase upon taking hypersurface section, reflecting the principle that singularities can only get worse after restriction to a hypersurface.

%This paper is organized as follows. We collect some preliminary materials in \S \ref{sec:prelim} and prove Theorem \ref{thm:adjunction} in \S \ref{sec:adjunction}. In \S \ref{sec:criteria} we prove a few criteria for the birational superrigidity and K-stability of complete intersections (including Theorem \ref{thm:main criterion}). Finally in \S \ref{sec:applications} we prove some lower bounds of minimal non-klt (resp. non-lc) colengths and as applications, prove the birational superrigidity and K-stability for a large class of Fano varieties.

This paper is organized as follows. We collect some preliminary materials in \S \ref{sec:prelim} and prove Theorem \ref{thm:main criterion} in \S \ref{sec:hyp-ordinary}. It is then applied to prove Theorem \ref{thm:hypersurface-ordinary} and Theorem \ref{thm:cpi-ordinary}. In \S \ref{sec:adjunction} we prove Theorem \ref{thm:adjunction} and finally Theorem \ref{thm:hypersurface-general} is proved in \S \ref{sec:hyp-general}.

\subsection*{Acknowledgement}
The first author would like to thank Chi Li and Chenyang Xu for helpful discussions. The second author would like to thank his advisor J\'anos Koll\'ar for constant support, encouragement and numerous inspiring conversations. He also wishes to thank Simon Donaldson for helpful comments and Charlie Stibitz for helpful discussion. We are also grateful to Kento Fujita, Fumiaki Suzuki and the anonymous referees for helpful comments. 

\section{Preliminary} \label{sec:prelim}

\subsection{Notation and conventions}

We work over $\bC$. By a pair $(X,D)$ we mean that $X$ is normal, $D$ is an effective $\bQ$-divisor, and $K_X+D$ is $\bQ$-Cartier. The notions of terminal, canonical, klt and log canonical (lc) singularities are defined in the sense of \cite[Definition 2.8]{mmp}. The tangent cone of a singularity is understood as the projective tangent cone throughout the note. Given a pair $(X,D)$ and an ideal sheaf $\fa$ (resp. a $\bQ$-Cartier divisor $\Delta$), the \emph{log canonical threshold} of $\fa$ (resp. $\Delta$) with respect to $(X,D)$ is denoted by $\lct(X,D;\fa)$ (resp. $\lct(X,D;\Delta)$). If $\fa$ is co-supported at a single closed point $x\in X$, we denote by $\e(\fa)$ the \emph{Hilbert-Samuel multiplicity} of $\fa$. We also denote $\e(x,X):=\e(\fm_x)$.

\subsection{Local invariants of singularities} \label{sec:local invariants}

\begin{defn}
 Let $x\in (X,D)$ be a klt singularity of dimension $n$.
 We define the \emph{local volume} $\hvol(x,X,D)$ and
 the \emph{normalized colength} $\hcol(x,X,D)$ of $x\in (X,D)$
 by
 \begin{align*}
  \hvol(x,X,D)&:=\inf_{\fa\colon\fm_x\textrm{-primary}}
  \lct(X,D;\fa)^{n}\cdot\e(\fa),\\
  \hcol(x,X,D)&:=\inf_{\fa\colon\fm_x\textrm{-primary}}
  \lct(X,D;\fa)^{n}\cdot\ell(\cO_{X,x}/\fa).
 \end{align*}
 If $x\in (X,D)$ is not klt, we set $\hvol(x,X,D)=\hcol(x,X,D)=0$. When $D=0$, we will simply write $\hvol(x,X)$ and $\hell(x,X)$.
\end{defn}

Note that our definition of local volume of a singularity
is equivalent to Li's original definition \cite{Li-normvol} in terms of valuations by \cite[Theorem 27]{Liu18}. A refined version
of normalized colengths was introduced in \cite{bl-semicont} in order to prove the lower semicontinuity of local volumes in families.

\begin{prop} \label{prop:lech}
 For any klt singularity  $x\in (X,D)$ of dimension $n$,
 we have
 \[
  \frac{1}{n!\e(x,X)}\hvol(x,X,D)\leq \hcol(x,X,D)\leq \frac{1}{n!}\hvol(x,X,D).
 \]
 If moreover $x$ is a smooth point on $X$, then 
 $\hvol(x,X,D)=n!\cdot\hcol(x,X,D)$.
\end{prop}

\begin{proof}
 The first inequality follows from Lech's inequality \cite[Theorem 3]{lech}
 \[n!\cdot\ell(\cO_{X}/\fa)\cdot\e(x,X)\geq \e(\fa).\]
 The second inequality follows from the fact that
 $n!\ell(\cO_{X}/\fa^m)=\e(\fa)m^n+O(m^{n-1})$.
\end{proof}

\begin{lem} \label{lem:compare colengths}
Let $x\in (X,D)$ be a klt singularity. Then
\[\lnklt(x,X,D)\ge \hell(x,X,D), \quad \lnlc(x,X,D)>\hell(x,X,D).
\]
\end{lem}

\begin{proof}
The minimal non-klt (resp. non-lc) colength of $x\in (X,D)$ is achieved by some ideal $\fa$. By definition, we have $\lct(X,D;\fa)\le 1$ (resp. $<1$), hence $\hell(x,X,D)\le \lct(X,D;\fa)^n\cdot \ell(\cO_X/\fa)\le$ (resp. $<$) $\ell(\cO_X/\fa) = \lnklt(x,X,D)$ (resp. $\lnlc(x,X,D)$).
\end{proof}

\subsection{Birational superrigidity} \label{sec:prelim-rigidity}

We refer to e.g. \cite[Definition 1.25]{Noether-Fano}) for the definition of birational superrigidity. We use the following equivalent characterization.

\begin{defn} \label{defn:movable boundary}
Let $(X,D)$ be a pair. A movable boundary on $X$ is defined as an expression of the form $a\cM$ where $a\in\bQ$ and $\cM$ is a movable linear system on $X$. Its $\bQ$-linear equivalence class is defined in an evident way. If $M=a\cM$ is a movable boundary, we say that the pair $(X,D+M)$ is klt (resp. canonical, lc) if for $k\gg 0$ and for general members $D_1,\cdots,D_k$ of the linear system $\cM$, the pair $(X,D+M_k)$ (where $M_k=\frac{a}{k}\sum_{i=1}^k D_i$) is klt (resp. canonical, lc) in the usual sense. For simplicity, we usually do not distinguish the movable boundary $M$ and the actually divisor $M_k$ for suitable $k$.
\end{defn}

\begin{thm}[{\cite[Theorem 1.26]{Noether-Fano}}] \label{thm:max singularity}
Let $X$ be a Fano variety. Then it is birationally superrigid if and only if it has $\bQ$-factorial terminal singularities, Picard number one, and for every movable boundary $M\sim_\bQ -K_X$ on $X$, the pair $(X,M)$ is canonical.
\end{thm}

\subsection{K-stability}

We refer to \cite{Tian-K-stability-defn,Don-K-stability-defn} for the definition of K-stability. We only need the following criterion in this note.

\begin{lem} \label{lem:rigid imply stable} \cite[Theorem 1.2]{rigid-imply-stable}
Let $X$ be a birationally superrigid Fano variety. Assume that $\lct(X;D)>\frac{1}{2}$ for every effective divisor $D\sim_\bQ -K_X$, then $X$ is K-stable.
\end{lem}

\section{Hypersurfaces with ordinary singularities} \label{sec:hyp-ordinary}

We start with the proof of Theorem \ref{thm:main criterion}, which relies on the following result:

\begin{lem} \label{lem:lct>1/2}
Let $X$ be a projective variety with klt singularities and $L$ an ample line bundle on $X$. Let $D\sim_\bQ L-K_X$ be a divisor on $X$ such that $(X,D)$ is log canonical outside a finite set $T$ of points. Assume that $h^0(X,L) < \lnklt(x,X)$ $($resp. $< \lnlc(x,X))$ for every $x\in T$, then $\lct(X;D)>\frac{1}{2}$ $($resp. $\ge \frac{1}{2})$.
\end{lem}

\begin{proof}
This is a special case of \cite[Theorem 3.3]{Z-cpi} (with $\Delta=0$ and $\lambda=1$).
\end{proof}

\begin{proof}[Proof of Theorem \ref{thm:main criterion}]
We first prove the birational superrigidity of $X$ under the assumptions (1) and (2). As $n\ge 2r+\delta+2\ge \delta+4$, the singular locus of $X$ has codimension at least $4$ and by \cite[Corollaire 3.14]{SGA2}, $X$ is factorial. We also have $\rho(X)=1$ since the cone over $X$ is also factorial by another application of \cite[Corollaire 3.14]{SGA2}. By assumption (2), for every $x\in X$, a general complete intersection of codimension $2r+\delta> \delta+1$ containing $x$ has klt singularities, hence by inversion of adjunction, $X$ has terminal singularities (note that if $x\in X$ is an isolated singularity and is not terminal, then a hyperplane section containing $x$ is not klt). By Theorem \ref{thm:max singularity}, it remains to show that the pair $(X,M)$ is canonical for every movable boundary $M\sim_\bQ -K_X$.

Let $M\sim_\bQ -K_X$ be a movable boundary. By \cite[Proposition 2.1]{Suzuki-cpi}, we have $\mult_S(M^2)\le 1$ for every subvariety $S$ of dimension $\ge 2r$ in the smooth locus of $X$. It follows that $\mult_x(M^2)\le 1$ outside a subset $Z$ of dimension $\le 2r+\delta$ (since a general complete intersection of codimension $\delta+1$ in $X$ is smooth). By \cite[Theorem 0.1]{dFEM-mult-and-lct}, $(X,2M)$ is log canonical outside $Z$. Let $x\in X$ and let $Y= X\cap \bP^{n-r-\delta}$ be a general linear space section of codimension $2r+\delta$ containing $x$, then $(Y,2M|_Y)$ is lc outside a finite set of points. We have $K_Y+2M|_Y\sim_\bQ (2r+\delta+1)H=L$ where $H$ is the hyperplane class. By Lemma \ref{lem:lct>1/2} and the assumption (2), we have $\lct(Y;2M|_D)\ge\frac{1}{2}$ since $h^0(Y,L)\le h^0(\bP^{n-r-\delta},\cO_{\bP^{n-r-\delta}}(2r+\delta+1))=\binom{n+r+1}{2r+\delta+1}$. Therefore, $(Y,M)$ is log canonical and by inversion of adjunction as before, $(X,M)$ has canonical singularities and hence $X$ is birationally superrigid.

Now assume that $X$ also satisfies (3). By Lemma \ref{lem:rigid imply stable}, $X$ is K-stable as long as $\lct(X;D)>\frac{1}{2}$ for every $D\sim_\bQ -K_X$. By \cite[Proposition 2.1]{Suzuki-cpi}, we have $\mult_S(D)\le 1$ for every subvariety $S$ of dimension $\ge r$ in the smooth locus of $X$ hence $\mult_x(D)\le 1$ and $(X,D)$ is log canonical outside a subset $Z$ of dimension $\le r+\delta$. Let $x\in X$ and let $Y$ be a general linear space section of codimension $r+\delta$ containing $x$, then $(Y,D|_Y)$ is lc outside a finite set of points. Let $L=(r+\delta)H\sim_\bQ K_Y+D|_Y$, then as 
\[h^0(Y,L)\le h^0(\bP^{n-\delta},\cO_{\bP^{n-\delta}}(r+\delta))=\binom{n+r}{r+\delta}<\lnklt(x,Y)
\]
by assumption (3), we get $\lct(Y,D|_Y)>\frac{1}{2}$ by Lemma \ref{lem:lct>1/2}, hence by inversion of adjunction we get $\lct(X;D)>\frac{1}{2}$ as desired.  
\end{proof}

Using the above criterion, we can now give the proof of Theorem \ref{thm:hypersurface-ordinary} and Theorem \ref{thm:cpi-ordinary}.

\begin{lem} \label{lem:non-klt colength}
Given $m,r\in\bZ_+$, then there exists a constant $N_0$ depending only on $m$ and $r$ such that for every ordinary canonical complete intersection singularity $x\in X$ of dimension $n\ge N_0$ and embedding codimension $r$ whose tangent cone has Fano index at least $m$, we have
\[\lnklt(x,X)\ge \binom{n-1}{m-1}, \quad \lnlc(x,X)\ge \binom{n-1}{m}.
\]
\end{lem}

\begin{proof}
For simplicity we only prove the inequality for minimal non-lc colength, since the other case is very similar. By \cite[Theorem A.2]{Z-cpi}, there exists a constant $N_0$ depending only on $m$ and $r$ such that for every smooth Fano complete intersection (in some $\bP^N$) of codimension $r$ and dimension at least $N_0-1$ and every movable boundary $M\sim_\bQ mH$ whose base locus has codimension at least $m+1$ (where $H$ is the hyperplane class), the pair $(X,M)$ is canonical. We will prove that the lemma holds under this choice of $N_0$. In other words, given $n\ge N_0$ and $X$ as in the statement of the lemma, we need to show that $\ell(\cO_X/\fa)\ge \binom{n-1}{m}$ for every $\fa$ co-supported at $x$ such that $(X,\fa)$ is not lc.

First notice that by the same degeneration argument as in \S \ref{sec:adjunction}, it suffices to prove this under the assumption that $X$ is the cone over a Fano complete intersection $V\subseteq\bP^{n-1+r}$ of codimension $r$ and index $s\ge m$ and the ideal $\fa$ is homogeneous. In this case, we have $X=\Spec(R)$ where
\[R=\bigoplus_{i=0}^\infty H^0(V,\cO_V(i))
\]
and $\fa\subseteq R$ is given by a graded system of linear series $\cM_\bullet=(\cM_i)$ where $\cM_i\subseteq H^0(V,\cO_V(i))$. Clearly $\ell(R/\fa)\ge h^0(V,\cO_V(i))-\dim \cM_i$ for all $i\ge 0$, hence the lemma would immediately follow once we prove the following claim:
    \begin{claim*}
    If $n\ge N_0$, then $h^0(V,\cO_V(m))-\dim \cM_m \ge \binom{n-1}{m}$.
    \end{claim*}
    
    To see this, let $Z$ be the base locus of $\cM_m$. Suppose that every component of $Z$ has codimension at least $m+1$, then since $\dim V=n-1\ge N_0-1$, the pair $(V,\cM_m)$ is canonical by the choice of $N_0$. As $V$ has Fano index at least $m$, $-(K_V+\cM_m)$ is nef, hence the cone over $(V,\cM_m)$ is lc by \cite[Lemma 3.1]{mmp}. In particular (since $\cM_m\subseteq\fa$), the pair $(X,\fa)$ is lc, contrary to our assumption on $\fa$. It follows that some irreducible component, say, $Z_0$ of $Z$ has codimension at most $m$. In other words, $\dim Z_0\ge n-1-m$.
    
    Let $\pi:V\dashrightarrow \bP^{n-1-m}$ be a general linear projection whose restriction to $Z_0$ is generically finite. Consider $\cN=f^*|\cO_{\bP^{n-1-m}}(m)|\subseteq H^0(V,\cO_V(m))$, then it is easy to see that $\dim \cN=\binom{n-1}{m}$, so the claim would follow if we have $\cM\cap\cN=\{0\}$. This last statement holds since every element of $\cM$ vanishes along $Z_0$ while by construction of $\pi$, none of the elements of $\cN$ (except zero) is identically zero along $Z_0$. We thus complete the proof of the claim and hence the lemma as well.
\end{proof}

\begin{proof}[Proof of Theorem \ref{thm:cpi-ordinary}]
Let $x\in X$ and let $Y\subseteq X$ be a general complete intersection of codimension $2r+\delta$ containing $x$, then the second assumption implies that $Y$ has smooth tangent cone of Fano index at least $2r+\delta+2$ at $x$. By Lemma \ref{lem:non-klt colength}, $\lnlc(x,Y)$ grows at least like a polynomial in $n$ of degree $2r+\delta+2$, hence for $n\gg 0$ we have $\lnlc(x,Y)> \binom{n+r+1}{2r+\delta+1}$ and therefore $X$ is birationally superrigid by Theorem \ref{thm:main criterion}. The proof of K-stability is similar.
\end{proof}

\begin{proof}[Proof of Theorem \ref{thm:hypersurface-ordinary}]
The existence of $N$ follows by taking $\delta=0$ and $r=1$ in the previous corollary. By \cite[Remark A.6]{Z-cpi}, we may take $N_0=200$ in \cite[Theorem A.2]{Z-cpi} for $m=4$, $r=1$. Hence $N\le 250$ by a careful inspection of the inequalities involved in the proof of Theorem \ref{thm:cpi-ordinary} (with $\delta=0$ and $r=1$).
\end{proof}

\section{Adjunction for local volumes and normalized colengths} \label{sec:adjunction}

In this section we prove Theorem \ref{thm:adjunction} and its stronger form Theorem \ref{thm:strongadj}, using a degeneration argument similar to \cite{lx16}.

For simplicity, we may assume $X=\Spec(R)$ is affine and $(X,D)$ is klt. In addition, we may assume that $K_X+D\sim_{\bQ} 0$ by shrinking $X$ if necessary. Let $\fm$ be the maximal ideal of $R$ whose cosupport is $x$. Let $H=(h=0)$ be a normal Cartier divisor on $X$ with $h\in\fm$. Denote $A:=R/(h)$ so that $H=\Spec\,A$. In this section, $H$ can be an irreducible component of $D$. Consider the extended Rees algebra $\cR:=\oplus_{k\in\bZ}\fa_k t^{-k}$
where $\fa_k:=\fa_k(\ord_H)=(h^{\max\{k,0\}})$.
Clearly, $\cR$ is a sub $\bC[t]$-algebra of $R[t,t^{-1}]$.
From \cite[Section 4]{lx16} we know that 
\[
 \cR\otimes_{\bC[t]}\bC[t,t^{-1}]\cong R[t,t^{-1}],
 \quad \cR\otimes_{\bC[t]}\bC[t]/(t)\cong \oplus_{k\in 
 \bZ_{\geq 0}}\fa_k/\fa_{k+1}=:T.
\]

For any $f\in R$, suppose $k=\ord_H(f)$. Then 
we define $\tilde{f}\in \cR$ to be the element $ft^{-k}\in \cR_k$.
Denote by $\bin(f):=[f]_{\fa_{k+1}}\in \fa_k/\fa_{k+1} = T_k$.
For an ideal $\fb$ of $R$, let $\fB$ be the ideal of $\cR$
generated by $\{\tilde{f}\colon f\in \fb\}$. Let
$\bin(\fb)$ be the ideal of $T$ generated by $\{\bin(f):f\in\fb\}$.
Let $\fm_T:=\bin(\fm)$, then it is clear that $\fm_T$ is a
maximal ideal of $T$.

\begin{lem}[{\cite[Lemma 4.1]{lx16}}]
 \begin{enumerate}
\item We have the identities:
\[
 \cR/\fB\otimes_{\bC[t]}\bC[t,t^{-1}]\cong (R/\fb)[t,t^{-1}], 
 \quad  \cR/\fB\otimes_{\bC[t]}\bC[t]/(t)\cong T/\bin(\fb);
\]
\item The $\bC[t]$-algebra $\cR/\fB$ is free and thus flat as a $\bC[t]$-module. In particular, we have the identity of dimensions:
\[\dim_{\bC}(R/\fb)=\dim_{\bC}(T/\bin(\fb));\]
\item If $\fb$ is $\fm$-primary, then $\bin(\fb)$ is an $\fm_T$-primary homogeneous ideal.
\end{enumerate}
\end{lem}

Since $\cR$ is a $\bC[t]$-algebra that is a flat $\bC[t]$-module, it induces a flat morphism $\pi:\cX=\Spec\, \cR\to \bA^1$. We know that  $\cX\setminus\cX_0=\pi^{-1}(\bA^1\setminus\{0\})\cong X\times (\bA^1\setminus\{0\})$. Let $\cD$ be the Zariski closure of $D\times(\bA^1\setminus\{0\})$. Then $(K_{\cX}+\cD)|_{\cX\setminus\cX_0}$ corresponds to the $\bQ$-Cartier $\bQ$-divisor $(K_X+D)\times(\bA^1\setminus\{0\})$ on $X\times(\bA^1\setminus\{0\})$. Since $K_X+D\sim_{\bQ}0$, we know that $K_{\cX}+\cD|_{\cX\setminus\cX_0}\sim_{\bQ}0$. Hence
$K_{\cX}+\cD$ is $\bQ$-linearly equivalent to a multiple of $\cX_0$. Since $\cX_0$ is Cartier on $\cX$, we have that $K_{\cX}+\cD$ is $\bQ$-Cartier on $\cX$. Hence $\pi:(\cX,\cD)\to\bA^1$ is a $\bQ$-Gorenstein flat family. 

Next we analyze the divisor $\cD_0$. Denote $D=cH+\sum_{i=1}^{l}c_i D_i$ where $c_i>0$ for each $1\leq i\leq l$.  
Let $\fp_i$ be the height $1$ prime ideal in $R$ corresponding to the generic point of $D_i$. Let $\cD_i$ be the Zariski closure of $D_i\times(\bA^1\setminus\{0\})$ in $\cX$. Then we know that the generic point $\cD_i$ corresponds to a height $1$ prime ideal $\fq_i$ of $\cR$, where
\[
\fq_i=\fp_i[t,t^{-1}]\cap\cR=\oplus_{k\in\bZ}(\fp_i\cap\fa_k) t^{-k}.
\]
Hence $\cD_{i,0}$ is the same as the divisorial part of the ideal sheaf $\fq_i|_{\cX_0}=\oplus_{k\in\bZ_{\geq 0}}
(\fa_k\cap\fp_i+\fa_{k+1})/\fa_{k+1}$ on $\cX_0=\Spec\,T$. Since $\fa_k=(h^{k})$ for $k\geq 0$ and $h\not\in\fp_i$, we know that $T\cong A[s]$ with $A=R/(h)$, and $\fa_k\cap\fp_i=(h^k)\fp_i$.
Under this isomorphism, $\fq_i|_{\cX_0}$ corresponds to 
$(\fp_i+(h)/(h))[s]$. Thus $\cD_{i,0}$ corresponds to $D_i|_{H}\times \bA^1$ under the isomorphism $\cX_0\cong H\times\bA^1$. Similarly, let $\cH$ be the Zariski closure of $H\times(\bA^1\setminus\{0\})$ in $\cX$. Let $\fq$ be the height $1$ prime ideal of $\cR$ corresponding to the generic point of $\cH$. We have 
\[
\fq=(h)[t,t^{-1}]\cap\cR=\oplus_{k\in\bZ} ((h)\cap\fa_k)t^{-k}.
\]
Hence $\cH_0$ is the same as the divisorial part of the ideal sheaf $\fq|_{\cX_0}=\oplus_{k\in\bZ_{\geq 0}}(\fa_k\cap(h)+\fa_{k+1})/\fa_{k+1}$ on $\cX_0$. Under the isomorphism $T\cong A[s]$, it is easy to see that $\fq|_{\cX_0}$ corresponds to the principal ideal $(s)$. Thus $\cH_0$ corresponds to $H\times\{0\}$ under the isomorphism $\cX_0\cong H\times\bA^1$. To summarize, we have shown the following proposition.

\begin{prop}\label{prop:centralfiber}
With the above notation, the family $\pi:(\cX,\cD)\to\bA^1$ is $\bQ$-Gorenstein and flat. Moreover, if $D=cH+\sum_{i=1}^l c_i D_i$ then
\[
(\cX_0,\cD_0)\cong (H\times\bA^1,c(H\times\{0\})+\sum_{i=1}^l c_i (D_i|_{H}\times\bA^1)).
\]
\end{prop}

Let $x_0\in\cX_0$ be the closed point corresponding to
$\fm_T$. Then it is clear that $\pi$ provides a special degeneration of $(x\in (X,D))$ to $(x_0\in (\cX_0,\cD_0))$. By Proposition \ref{prop:centralfiber},
there is a natural $\bG_m$-action on $(x_0\in(\cX_0,\cD_0))$ induced by the standard $\bG_m$-action on $\bA^1$. We define the \emph{$\bG_m$-invariant local volume/normalized colength} of $(x_0\in (\cX_0,\cD_0))$ as
\begin{align*}
\hvol^{\bG_m}(x_0,\cX_0,\cD_0)&:=\inf\{\lct(\cX_0,\cD_0;\fa)^n\cdot\e(\fa)\mid \fa\textrm{ is $\fm_{T}$-primary and $\bG_m$-invariant}\};\\
\hcol^{\bG_m}(x_0,\cX_0,\cD_0)&:=\inf\{\lct(\cX_0,\cD_0;\fa)^n\cdot\ell(T/\fa)\mid \fa\textrm{ is $\fm_{T}$-primary and $\bG_m$-invariant}\}.
\end{align*}
By \cite[Section 4]{lx16}, our definition of  $\bG_m$-invariant local volume is the same as the infimum of normalized volumes of all $\bG_m$-invariant valuations.

The following lemma is similar to \cite[Lemma 4.3 and 4.4]{lx16}.
\begin{lem} \label{lem:degeneration}
With the above notation, let $\fb$ be an $\fm$-primary ideal of $R$. Then 
 \[
  \lct(X,D;\fb)\geq \lct(\cX_0,\cD_0,\bin(\fb)),\quad
  \ell(R/\fb)=\ell(T/\bin(\fb)).
 \]
 Moreover,
 \begin{align*}
\hcol(x,X,D)&\geq \hcol^{\bG_m}(x_0,\cX_0,\cD_0)=\hcol(x_0,\cX_0,\cD_0),\\
\hvol(x,X,D)&\geq \hvol^{\bG_m}(x_0,\cX_0,\cD_0)=\hvol(x_0,\cX_0,\cD_0).
 \end{align*}

\end{lem}

\begin{proof}
The inequalities on log canonical thresholds follows from the lower semi-continuity of $\lct$ \cite[Corollary 1.10]{ambro-lct}.
Thus for any $\fm$-primary ideal $\fb$, we have
\[
\lct(X,D;\fb)^n\cdot\ell(R/\fb)\geq \lct(\cX_0,\cD_0;\bin(\fb))^n\cdot\ell(T/\bin(\fb)).
\]
Thus we have $\hcol(x,X,D)\geq\hcol^{\bG_m}(x_0,\cX_0,\cD_0)$. Apply this initial degeneration argument to ideals in $T$, we get $\hcol(x_0,\cX_0,\cD_0)\geq \hcol^{\bG_m}(x_0,\cX_0,\cD_0)$. By definition we have $\hcol(x_0,\cX_0,\cD_0)\leq \hcol^{\bG_m}(x_0,\cX_0,\cD_0)$, hence the inequality on normalized colengths is proved. The last inequality on local volumes is proved in the same way as \cite[Lemma 4.3]{lx16} using graded sequence of ideals.
\end{proof}

Next we compare the local volume and normalized colength
of $x_0\in (\cX_0,\cD_0)$ to $x\in(H,(D-cH)|_H)$.

\begin{lem} \label{lem:product}
 With the above notation, we have
 \begin{align*}
\hcol(x_0,\cX_0,\cD_0)&\geq 
  \left(\frac{n}{n-1}\right)^{n-1}(1-c)\hcol(x,H,(D-cH)|_H),\\
  \hvol(x_0,\cX_0,\cD_0)& =
  \frac{n^n}{(n-1)^{n-1}}(1-c)\hvol(x,H,(D-cH)|_H).
 \end{align*}
\end{lem}

\begin{proof}
If $(\cX_0,\cD_0)$ is not klt at $x_0$, then $(H,(D-cH)|_H)$ is not klt at $x$ by inversion of adjunction, in which case the lemma is trivial. Thus we may assume that $(\cX_0,\cD_0)$ is klt at $x_0$, and $(H,(D-cH)|_H)$ is klt at $x$.
By Proposition \ref{prop:centralfiber}, we have
\[
(x_0\in (\cX_0,\cD_0))\cong ((x,0)\in (H\times\bA^1,c(H\times\{0\})+(D-cH)|_H\times\bA^1))
\]
We need to use a lemma on $\bG_m$-equivariant valuations.
Assume $H=\Spec\,A$ is affine. Let $s$ be the parameter of
$\bA^1$ in $H\times\bA^1$.

\begin{lem}[{\cite[Section 11]{aipsv} and \cite[Theorem 2.15]{lx17}}]
 For any valuation $v$ on $\bC(H)$
 and any $\xi\in\bR$,  we can define a valuation
 $\tilde{v}_\xi$ on $\bC(H\times\bA^1)$ as follows:
 \[
 \tilde{v}_{\xi}(f):=\min_{0\leq i\leq m,~f_i\neq 0}\{v(f_i)+\xi\cdot i\}\quad\textrm{for 
 any }f=\sum_{i=0}^m f_i s^i\in A[s].
 \]
 Conversely, any $\bG_m$-invariant valuation on $\bC(H\times\bA^1)$
 is of the form $\tilde{v}_\xi$ for some $v\in\Val_H$ and $\xi\in\bR$.
 Moreover, $\tilde{v}_{\xi}$ is centered at $(x,0)$ if and 
 only if $v$ is centered at $x$ and $\xi>0$.
\end{lem}

Now assume $\fb$ is an ideal sheaf on $H\times\bA^1$ cosupported
at $(x,0)$, then we may write $\fb$ as
\[
 \fb=\fb_0+\fb_1 s+\cdots+\fb_m s^m +(s^{m+1})\subset A[s].
\]
For simplicity, we will not distinguish 
$(\cX_0,\cD_0)$ from $(H\times\bA^1,c(H\times\{0\})+(D-cH)|_H\times\bA^1)$.
Assume $\lct(\fb):=\lct(\cX_0,\cD_0,\fb)$ is computed by a $\bG_m$-invariant divisorial
valuation $\tilde{v}_\xi$ (for some valuation $v$ on $\bC(H)$ and some $\xi\in\bR$). By \cite[Lemma 26]{Liu18}, $\tilde{v}_\xi$ is centered at $(x,0)$.
Denote $k:=\tilde{v}_{\xi}(\fb)$
and $a:=A_{(H,(D-cH)|_H)}(v)$, then $A_{(H,(D-cH)|_H)\times\bA^1}(\tilde{v}_\xi)=a+\xi$ (by \cite[\S 5]{JM-valuation}, it suffices to prove this when $v$ is quasi-monomial, in which case $\tilde{v}_\xi$ is also quasi-monomial and the result is clear). 
Since $\cD_0=cH\times\{0\}+(D-cH)|_H\times\bA^1$, we have 
\[
\lct(\fb)=\frac{A_{(\cX_0,\cD_0)}(\tilde{v}_\xi)}{\tilde{v}_\xi(\fb)}=\frac{A_{(H,(D-cH)|_H)\times\bA^1}(\tilde{v}_\xi)-\tilde{v}_\xi(cH\times\{0\})}{\tilde{v}_\xi(\fb)}
=\frac{a+(1-c)\xi}{k}.
\]
From the definition of $\tilde{v}_\xi$ we see that
$v(\fb_i)+\xi\cdot i\geq \tilde{v}_\xi(\fb)=k$.
Hence $$\lct(\fb_i):=\lct(H,(D-cH)|_H;\fb_i)\leq \frac{a}{v(\fb_i)}\leq
\frac{a}{k-\xi\cdot i} \textrm{ for any }i<\frac{k}{\xi}.$$
We know that $\lct(\fb_i)^{n-1}\cdot\ell(A/\fb_i)\geq \hcol(x,H,(D-cH)|_H)$.
This implies for any $i<\frac{k}{\xi}$ we have
\[
 \ell(A/\fb_i)\geq \frac{\hcol(x,H,(D-cH)|_H)}{\lct(\fb_i)^{n-1}}
 \geq \frac{\hcol(x,H,(D-cH)|_H)}{a^{n-1}}(k-\xi\cdot i)^{n-1}.
\]
Since $s^{m+1}\in\fb$, we have that 
$k=\tilde{v}_{\xi}(\fb)\leq \tilde{v}_{\xi}(s^{m+1})=\xi(m+1)$, i.e.
$\lceil k/\xi\rceil-1\leq m$.
Thus
\begin{align*}
 \ell(A[s]/\fb)&\geq\sum_{i=0}^{\lceil k/\xi\rceil-1} \ell(A/\fb_i)
 \geq \frac{\hcol(x,H,(D-cH)|_H)}{a^{n-1}}\sum_{i=0}^{\lceil k/\xi\rceil-1}(k-\xi\cdot i)^{n-1}\\
 & \geq \frac{\hcol(x,H,(D-cH)|_H)}{a^{n-1}}\int_{0}^{k/\xi}(k-\xi\cdot w)^{n-1}dw
 \\& =\frac{k^n}{na^{n-1}\xi}\hcol(x,H,(D-cH)|_H).
\end{align*}
Therefore,
\begin{align*}
 \lct(\fb)^n\cdot\ell(A[s]/\fb)&\geq \left(\frac{a+(1-c)\xi}{k}\right)^n
 \frac{k^n}{na^{n-1}\xi}\hcol(x,H,(D-cH)|_H)\\&=\frac{(a+(1-c)\xi)^n}{na^{n-1}\xi}\hcol(x,H,(D-cH)|_H)\\
 &\geq\left(\frac{n}{n-1}\right)^{n-1}(1-c)\hcol(x,H,(D-cH)|_H)
\end{align*}
where the last inequality follows from the AM-GM inequality. This proves the inequality on normalized colengths.

For the equality on local volumes,
note that (see e.g. \cite{Li-normvol} for the definition of log discrepancy, volume and normalized volume of a valuation)
\[
 A_{(\cX_0,\cD_0)}(\tilde{v}_{\xi})=a+(1-c)\xi, 
 \quad \vol(\tilde{v}_{\xi})=\frac{\vol(v)}{\xi}.
\]
Hence again by the AM-GM inequality,
\[
 \hvol(\tilde{v}_{\xi})=\frac{(a+(1-c)\xi)^n}{\xi}\vol(v)
 \geq \frac{n^n(1-c)}{(n-1)^{n-1}}a^{n-1}\vol(v)=\frac{n^n(1-c)}{(n-1)^{n-1}}\hvol(v).
\]
On the other hand, we have 
\[
\hvol(\tilde{v}_{\frac{a}{(1-c)(n-1)}})=
\frac{n^n(1-c)}{(n-1)^{n-1}}\hvol(v).
\]
Thus the equality is also proved.
\end{proof}

Combining Lemma \ref{lem:degeneration} and \ref{lem:product}, we have the following theorem which yields Theorem \ref{thm:adjunction} when $c=0$.

\begin{thm}\label{thm:strongadj}
Let $x\in(X,D)$ be a klt singularity of dimension $n$. Let $H$ be a normal reduced Cartier divisor of $X$ containing $x$. Let $c$ be the coefficient of $H$ in $D$. Then we have
\[
  \frac{\hvol(x,X,D)}{n^n}\geq (1-c)\frac{\hvol(x,H,(D-cH)|_H)}{(n-1)^{n-1}},
  \quad \frac{\hcol(x,X,D)}{n^n/n!}\geq (1-c)\frac{\hcol(x,H,(D-cH)|_H)}{(n-1)^{n-1}/(n-1)!}.
\]
\end{thm}

The equality of local volumes in Lemma \ref{lem:product} suggests a conjectural product formula for local volumes as follows.

\begin{conj}\label{conj:product}
Let $(x_i\in (X_i,D_i))$ be klt singularities for $i=1,2$. Denote $n_i:=\dim X_i$, then 
\[
\frac{\hvol((x_1,x_2),X_1\times X_2,\pi_1^* D_1+\pi_2^* D_2)}{(n_1+n_2)^{n_1+n_2}}=\frac{\hvol(x_1,X_1,D_1)}{n_1^{n_1}}\cdot\frac{\hvol(x_2,X_2,D_2)}{n_2^{n_2}}.
\]
\end{conj}

\begin{rem} Let us make some remarks on Conjecture \ref{conj:product}.
\begin{enumerate}[label=(\alph*)]
    \item  If each singularity $(x_i\in X_i)$ lives on a Gromov-Hausdorff limit of K\"ahler-Einstein Fano manifolds, then by \cite[Corollary 5.7]{lx17} we know that 
    \[\hvol(x_i,X_i)=n_i^{n_i}\Theta(x_i,X_i)\]
    where $\Theta(\cdot,\cdot)$ is the volume density of a singularity (see e.g. \cite{hs16, Spotti-Sun}). It is well known that $\Theta((x_1,x_2),X_1\times X_2)=\Theta(x_1,X_1)\cdot\Theta(x_2,X_2)$, so Conjecture \ref{conj:product} holds in this case. More generally, similar arguments (e.g. \cite[Theorem 1.5]{lx17}) show that Conjecture \ref{conj:product} is true if each singularity $(x_i\in X_i)$ admits a special degeneration to a Ricci-flat K\"ahler cone singularity.
    \item Our method in this section should be enough to prove special cases of Conjecture \ref{conj:product} when one of the klt singularities is toric.
    \item It is also natural to expect the following statement to be true: let $v_i$ be a $\hvol$-minimizing valuations over $(x_i\in (X_i,D_i))$, then there exists a $\hvol$-minimizer $v$ over $((x_1,x_2)\in (X_1\times X_2,\pi_1^*D_1+\pi_2^*D_2))$ such that for any function $f_{i,j}\in\bC(X_i)^{\times}$, 
\[
 v(\sum_j f_{1,j}\otimes f_{2,j})=\min_j\left(\frac{n_1 v_1(f_{1,j})}{A_{(X_1,D_1)}(v_1)}+\frac{n_2 v_2(f_{2,j})}{A_{(X_2,D_2)}(v_2)}\right).
\]
\end{enumerate}
\end{rem}

\begin{psr}
After the first version of this article was posted on the arXiv, Li, Xu and the first author proved an inversion of adjunction type result on local volumes of plt pairs \cite[Proposition 6.8]{LLX18} based on Theorem \ref{thm:strongadj}.
\end{psr}

\section{Proof of Theorem \ref{thm:hypersurface-general}} \label{sec:hyp-general}

In this section we prove Theorem \ref{thm:hypersurface-general}. Since the minimal non-klt (or non-lc) colength of a singularity is bounded from below by its normalized colength, we have the following criterion by combining Theorems \ref{thm:main criterion} and Theorem \ref{thm:adjunction}:

\begin{cor} \label{cor:colength criterion}
Let $\delta\ge -1$ be an integer and let $X\subseteq\bP^{n+r}$ be a Fano complete intersection of index $1$, codimension $r$ and dimension $n$. Assume that
    \begin{enumerate}
        \item The singular locus of $X$ has dimension at most $\delta$ and $2n+1\ge 3(r+\delta)$;
        \item For every $x\in X$ and every general linear space section $Y\subseteq X$ of codimension $2r+\delta$ containing $x$, the normalized colength of $(Y,x)$ satisfies $\hell(x,Y)\ge \binom{n+r+1}{2r+\delta+1}$.
    \end{enumerate}
Then $X$ is birationally superrigid and K-stable.
\end{cor}

\begin{proof}
By Theorem \ref{thm:main criterion} and Lemma \ref{lem:compare colengths}, it suffices to show that
\[\hell(x,Y)> \binom{n+r}{r+\delta}
\]
for every $x\in X$,  where $Y\subseteq X$ is a general linear section of codimension $r+\delta$ containing $x$. Let $Y'\subseteq Y$ be a general complete intersection of codimension $r$ in $Y$ (in particular, $Y'$ is a general linear section of codimension $2r+\delta$ in $X$). By Theorem \ref{thm:adjunction}, we have 
\[\hell(x,Y)\ge \left( \frac{m}{m-1} \right)^{m-1} \hell(x,H) \ge 2\hell(x,H)
\]
where $m=\dim Y$ and $H\subseteq Y$ is a general hyperplane section containing $x$, hence by our second assumption and a repeated use of the above inequality,
\begin{eqnarray*}
    \hell(x,Y) & \ge & 2^r \hell(x,Y') \; \ge \; 2^r \binom{n+r+1}{2r+\delta+1} \\
               & >   & 2^r \binom{n+r}{2r+\delta} \; = \; \binom{n+r}{r+\delta} \cdot \prod_{i=0}^{r-1} \frac{2(n-i-\delta)}{2r-i+\delta} \\
               & \ge & \binom{n+r}{r+\delta},
\end{eqnarray*}
where the last inequality holds as $2n\ge 3r+3\delta-1\ge 2r+3\delta+i$. This completes the proof.
\end{proof}

Theorem \ref{thm:hypersurface-general} can now be deduced from the following lower bounds of local volumes and normalized colengths.

\begin{lem} \label{lem:colength-bounds}
Let $x\in X$ be a hypersurface singularity of multiplicity $m\ge 2$ and dimension $n\ge m$. Assume that the tangent cone of $x\in X$ is smooth in codimension $m-1$, then 
\[\hvol(x,X)\ge \frac{n^n}{m^{m-1}},\quad \hell(x,X)\ge \frac{n^n/n!}{m^m}.
\]
\end{lem}

\begin{proof}
It suffices to prove the first inequality since the second follows from the first by Proposition \ref{prop:lech}. Let $Y=X\cap H_1 \cap H_2 \cap \cdots \cap H_{n-m}$ be a general complete intersection of dimension $m$ containing $x$, then we have $\rme(x,Y)=\rme(x,X)=m$ and by our assumption on the tangent cone of $X$, we see that $x\in Y$ has a smooth tangent cone $V$ given by a smooth hypersurface of degree $m$ in $\bP^m$. It is clear that there exists a $\bQ$-Gorenstein flat family over $\mathbb{A}^1$ specially degenerating $Y$ to the affine cone $C_p(V)$ over $V$ (with vertex $p$), hence by \cite[Theorem 1]{bl-semicont}, we have $\hvol(x,Y)\ge \hvol(p,C_p(V))$. On the other hand, $V$ is K-stable by \cite[Theorem 1.2 and Example 1.3(2)]{Fujita-alpha}, thus by \cite[Corollary 1.5]{LL-vol-min}, the local volume of $p\in C_p(V)$ is computed by the blowup of $p$ and we get $\hvol(p,C_p(V))=m$. Combining these with Theorem \ref{thm:adjunction} we obtain:
\[\frac{\hvol(x,X)}{n^n}\ge \frac{\hvol(x,X\cap H_1)}{(n-1)^{n-1}}\ge \cdots \ge \frac{\hvol(x,Y)}{m^m} \ge \frac{\hvol(p,C_p(V))}{m^m} = \frac{1}{m^{m-1}},
\]
giving the required statement of the lemma.
\end{proof}

\begin{proof}[Proof of Theorem \ref{thm:hypersurface-general}]
The condition (2) is preserved by taking general hyperplane sections containing a given point, hence for every $x\in X$ and every general complete intersection $Y\subseteq X$ of codimension $2+\delta$ containing $x$, we have 
\[\hell(x,Y)\ge \frac{1}{m^m}\cdot \frac{(n-2-\delta)^{n-2-\delta}}{(n-2-\delta)!} \ge \binom{n+2}{\delta+3}
\]
by Lemma \ref{lem:colength-bounds} and condition (3). Hence $X$ satisfies both assumptions of Corollary \ref{cor:colength criterion} (with $r=1$) and the result follows.
\end{proof}

If the hypersurface $X\subseteq \bP^{n+1}$ has isolated ordinary singularities (i.e. the tangent cone is smooth), then one can usually expect a better lower bound of its corresponding local volume. Indeed, we should have $\hvol(x,X)\ge m(n+1-m)^n$ (where $m=\rme(x,X)$) by \cite[Corollary 1.5]{LL-vol-min} and the conjectural K-stability of smooth Fano hypersurfaces. For small multiplicities, this usually give better results than Theorems \ref{thm:hypersurface-ordinary} and Theorem \ref{thm:hypersurface-general}. For instance, in the case of $m=2$ we do have $\hvol(x,X)\ge 2(n-1)^n$ (quadric hypersurfaces admit K\"ahler-Einstein metrics as they are homogeneous), hence:

\begin{cor}
Let $X\subseteq\bP^{n+1}$ be a hypersurface of degree $n+1$. Assume that $X$ has at most ordinary double points, then $X$ is birationally superrigid and K-stable when $n\ge 11$.
\end{cor}

\begin{proof}
Take $\delta=0$ and $r=1$ in Corollary \ref{cor:colength criterion}. By the above remark and Proposition \ref{prop:lech}, for every general complete intersection $x\in Y\subseteq X$ of codimension $2$ we have
\[\hell(x,Y)\ge \frac{1}{2(n-2)!} \hvol(x,Y) = \frac{(n-3)^{n-2}}{(n-2)!}.
\]
Hence by Corollary \ref{cor:colength criterion}, such hypersurface $X$ is birationally superrigid and K-stable as long as $\frac{(n-3)^{n-2}}{(n-2)!}\ge \binom{n+2}{3}$, or equivalently, $n\ge 11$.
\end{proof}

\bibliography{ref}
\bibliographystyle{alpha}

\end{document}